\documentclass[letterpaper,10pt,conference]{ieeeconf}
\IEEEoverridecommandlockouts
\makeatletter

\let\proof\@undefined
\let\endproof\@undefined
\makeatother
\usepackage[dvips]{graphicx}
\usepackage{graphicx,psfrag}
\usepackage[noend]{algorithmic}
\usepackage{xspace}
\usepackage{amssymb,amsmath,amscd,mathrsfs}
\usepackage{amsthm}
\usepackage{comment,color}
\usepackage{subfig}
\usepackage{url}
\usepackage{multirow}

\usepackage{graphicx}
\usepackage{psfrag}

\usepackage[bookmarks=true]{hyperref}
\usepackage{fixltx2e}

\newcommand{\BEAS}{\begin{eqnarray*}}
\newcommand{\EEAS}{\end{eqnarray*}}
\newcommand{\BEQ}{\begin{equation}}
\newcommand{\EEQ}{\end{equation}}
\newcommand{\BIT}{\begin{itemize}}
\newcommand{\EIT}{\end{itemize}}

\newcommand{\dist}{\mathop{\bf dist{}}}

\def\defeq{\triangleq}

\newcommand{\probj}{\mathbb{P}}

\usepackage{color}
\usepackage{amsmath}
\usepackage{amssymb}
\usepackage{graphicx}
\usepackage{comment,xspace}
\usepackage{fancybox}

\newcommand{\real}{{\mathbb{R}}}
\newcommand{\reals}{\real}

\renewcommand{\natural}{{\mathbb{N}}}
\newcommand{\naturals}{\natural}

\newtheorem{theorem}{Theorem}[section]

\newtheorem{remark}[theorem]{Remark}

\newtheorem{definition}[theorem]{Definition}

\newtheorem{assumption}[theorem]{Assumption}

\newcommand{\expectation}[1]{\mbox{$\mathbb{E}\left[#1\right]$}}

\def\S{\mathcal{S}}
\def\A{\mathcal{A}}
\def\R{R}
\def\P{P}

\def\V{V}
\def\Q{Q}

\def\pb{\rho} 
\def\icr{\beta} 
\def\dist{D_\pb} 
\def \df{\alpha} 
\def \epmt{\tau_\epsilon} 
\def \oepmt{\epmt^\star} 
\def\dconst{C_D} 
\def\aconst{C_A} 
\def\ctheta{C_\theta} 

\def \y{y_0} 
\def \hist{\mathcal{H}} 

\newcommand{\BEF}{\begin{proof}}
\newcommand{\EEF}{\end{proof}}

\title{\LARGE \bf
Weighted Difference Approximation of Value Functions \\  for Slow-Discounting Markov Decision Processes
}

\author{Yin-Lam Chow and Junjie Qin
\thanks{Y.-L. Chow and J. Qin are with the Institute for Computational and Mathematical Engineering, Stanford University, Stanford, CA 94305, USA. Email: {\tt \{ychow, jqin\}@stanford.edu}.}}

\begin{document}

\maketitle
\thispagestyle{empty}
\pagestyle{empty}

\begin{abstract}
Modern applications of the theory of Markov Decision Processes (MDPs) often require frequent decision making, that is, taking an action every microsecond, second, or minute. Infinite horizon discount reward formulation is still relevant for a large portion of these applications, because actual time span of these problems can be months or years, during which discounting factors due to e.g. interest rates are of practical concern. In this paper, we show that, for such MDPs with discount rate $\alpha$ close to $1$, under a common ergodicity assumption, a weighted difference between two successive value function estimates obtained from the classical value iteration (VI) is a better approximation than the value function obtained directly from VI. Rigorous error bounds are established  which in turn show that the approximation converges to the actual value function in a rate $(\alpha \beta)^k$ with $\beta<1$. This indicates a geometric convergence even if discount factor $\alpha \to 1$. Furthermore, we explicitly link the convergence speed to the system behaviors of the MDP using the notion of $\epsilon-$mixing time and extend our result to Q-functions. Numerical experiments are conducted to demonstrate the convergence properties of the proposed approximation scheme.
%
\end{abstract}
%

\addtolength{\topskip}{-.0in}
\addtolength{\abovedisplayskip}{-.06in}
\addtolength{\belowdisplayskip}{-.06in}
\addtolength{\parskip}{-0.02in}
\addtolength{\belowcaptionskip}{-0.1in}
\section{Introduction}
A large number of practical problems that involved with decision making under uncertainty can be modeled as Markov Decision Problems (MDPs). Among them, many with relatively long planning horizons are suitably casted as infinite horizon MDPs, with either discounted reward or average reward criteria \cite{puterman_05}. While discounted reward formulation features easier-to-implement computational methods such as value iteration, in cases where the discount factor is very close to $1$, it is known that the convergence for the discounted reward value iteration can be unacceptably slow. This occurs for example in communication network and computer systems applications where decisions have to be made frequently. The average reward criteria, together with their theoretical analysis and algorithmic development, were in part motivated by these observations. However, for these slow-discounting problems, the approach of first modeling the problem approximately as an average reward MDP and then solving it with corresponding algorithms (cf. Chapter~5 of \cite{bertsekas_12} for more details) may give a suboptimal policy with respect to the original discounted reward criteria. 

This paper provides a scheme for approximating value functions of slow-discounting MDPs. The approximation is in the form of a weighted difference between two successive value function estimates obtained from the classical VI.
In particular, building from theories connecting the average reward criteria and discounted reward criteria, we demonstrate that the approximation has a geometric convergence with an error bound of the order $(\df \beta)^k$ which approaches zero even when $\df\to 1$, where $\beta <1$ under a common ergodicity assumption and $k$ is the iteration count for VI. The rate parameter $\beta$ is then characterized with the well-understood notion of $\epsilon$-mixing time for average reward problems. 

The contributions of this paper are summarized as follows: 
\begin{itemize}
\item We show that using a weighted difference between two successive iterates, the classical VI algorithm can be made practical even if the discount factor is arbitrarily close to one.
\item We characterize the convergence of such value function approximation and discuss its relation to the notation of $\epsilon-$mixing time. The error bounds for the value function approximation provides novel insights on the discounted Bellman operator for ergodic MDPs, and theoretical backups for learning algorithms which may need to solve slow-discounting MDPs in its iterates\footnote{For example, the polynomial sample complexity bounds for reinforcement learning algorithm proposed in \cite{Kearns2002} will not be meaningful if $\df\to 1$  and if classical VI is used for solving the MDP in each step. }. 
\item We extend the above weighted difference approximation scheme to Q-functions, which is more commonly used in many reinforcement learning algorithms.
\end{itemize}
 
 \subsection{Related Literature}
Several methods have been proposed for solving MDPs with discount factor $\df$ close to $1$. Among them,  splitting methods and relative value iteration (RVI) are well studied. The Gauss-Seidel VI is the most noteworthy example of splitting methods \cite{puterman_05}, which has $(\df\beta^\mathrm{GS})^k$ convergence, where $\df\beta^\mathrm{GS}$ is related to the norm of corresponding splitting matrices. However the $\beta^\mathrm{GS}<1$ term is usually difficult to evaluate in general settings. In Section~\ref{sec:sim}, the performance of our approximation scheme and Gauss-Seidel VI is compared numerically.
The RVI algorithm, proposed by \cite{white1963dynamic} for average reward problems and generalized to discounted reward settings by \cite{macqueen1966modified} and \cite{Odoni1969}, is shown to have a $(\df \beta^\mathrm{RVI})^k$  convergence in \cite{Morton1971}. The convergence is proved in terms of the relative value function, which is the difference between the value of each state and the value of a fixed pre-selected state, and $\beta^\mathrm{RVI}<1$ is the second largest eigenvalue of the transition probability matrix corresponding to the optimal policy.  While both the RVI and our approximation scheme are analyzed under a similar ergodicity assumption, we contrast these two approaches as follows:
\begin{itemize}
\item RVI is constructed to be an algorithm to obtain the relative value function, which provides sufficient information to compute the optimal policy. However,  to get the actual value function, one has to perform one-step policy evaluation after the algorithm converges, which requires solving a large linear system when the number of state is tremendous. Our approximation scheme estimates the value function directly, which is superior to RVI in applications such as hybrid systems where the actual value functions for each subsystem are often needed for comparison.
\item The $\beta^\mathrm{RVI}$ term in the convergence rate of RVI is hard to evaluate ahead of solving the problem since it corresponds to the optimal policy. Our convergence rate can be obtained directly from the problem data beforehand.
\item Our approach is both conceptually and implementation-wise simpler as its major computation is merely the classical VI. 
\end{itemize}

\subsection{Paper Organization}
The rest of the paper is organized as follows. Section \ref{sec:setup} introduces the problem setup and definitions used. The approximation scheme based on weighted difference is provided in Section \ref{sec:VI_conv}, followed by a proof on its error bound. A characterization for the rate parameter $\beta$ is derived in Section \ref{sec:mix_time}, based on a connection to the concept of  $\epsilon-$mixing time. 
The results of a numerical experiment are given in Section \ref{sec:sim}. Finally, this paper concludes with Section \ref{sec:conc}. 

\section{Problem Setup}\label{sec:setup}
Consider an infinite horizon discounted MDP characterized by the quintuple $(\S,\A,\R,\P,\df)$. Here $\S$ and $\A$ are finite sets representing the state space and the action space. For each $(x,a,y)\in \S\times \A \times \S$, $\R_a(x,y)\in [0,\R_{\max}]$ and $\P_a(x,y)\in [0,1]$ are reward and probability of transitioning from state $x$ to state $y$ after taking action $a$, respectively. The discount rate is denoted as $\df\in (0,1)$. In standard MDPs, the agent aims to identify a stationary policy $\mu: \S\rightarrow \A$ that maximizes the expected discounted reward
\[
\mathbb E \left[\sum_{t=1}^\infty \df^t R_{\mu(x_t)} (x_t, x_{t+1}) \right].
\]

Starting from each state $x\in \S$, the $N$-step accumulated discounted reward for policy $\mu$ is defined as
\[
    \V^N_\mu(x) =\mathbb E \left[\sum_{t=1}^N \df^t R_{\mu(x_t)} (x_t, x_{t+1}) \middle| x_0 = x\right].
\]
By Monotone Convergence Theorem, the infinite horizon value function with respect to control policy $\mu$ is given by
\[
\!\V_\mu(x)=\lim_{N\rightarrow\infty}\V^N_\mu(x)=\mathbb E \left[\sum_{t=1}^\infty \df^t R_{\mu(x_t)} (x_t, x_{t+1}) \middle| x_0 = x\right]
\]
and the (optimal) value function is defined by
$    \V^\star (x) \defeq \max_{\mu } \V_\mu(x).$
Similarly, we can define the state-action value function for each state-action pair $(x,a)\in \S\times \A$ and policy $\mu$ as
$\Q_\mu(x,a) = \sum_{y\in \S} \P_a(x,y) ( \R_a(x,y)+\df \V_\mu (y)),$
and the optimal $\Q-$function as
\begin{equation}\label{Bellman_Q}
    \Q^\star (x,a) = \sum_{y\in \S} \P_a(x,y) (\R_a(x,y)+\df \V^\star (y)).
\end{equation}
Note that $\V^\star(x)=\max_{a\in\A}\Q^\star(x,a)$ is the value function that satisfies the Bellman equation: $\V^\star(x)=T[\V^{\star}](x)$, for every $x\in\S$. The Bellman operator for discounted reward function is denoted by $T[\cdot]$, where
\begin{equation}
T[\V](x)\defeq\max_{a\in\A} \sum_{y\in\S} \P_a(x,y)(R_a(x,y)+\df\V(y)).\label{eq:T}
\end{equation}
for $\df\in(0,1)$, and $V:\S\rightarrow\reals$ is an arbitrary function. We can write expression (\ref{Bellman_Q}) as the Bellman equation of optimal $Q-$function:
$\Q^\star(x,a)=F[\Q^\star](x,a), \quad \forall x\in\S, a\in\A, $
where $F[\cdot]$ is the $Q-$function Bellman operator, defined as
\[
F[\Q](x,a)= \sum_{y\in\S} P_a(x,y)(R_a(x,y)+\df\max_{b\in\A} \Q(y,b)),
\]
for $\df\in(0,1)$, and $Q:\S\times\A\rightarrow\reals$ is an arbitrary function. Furthermore, let $\mu_\Q$ be a policy which satisfies
$    \mu_\Q \in \arg\max_{a\in \A} \Q(x,a).$

Ergodicity assumptions are widely used in the analysis of stochastic optimal control and reinforcement learning \cite{Kearns2002, Brafman2003}. Motivated by identical assumptions made in the analysis of the relative value iteration algorithm for average reward MDPs (cf. Proposition 5.3.2 in \cite{bertsekas_12}), we give a more quantitative characterization of the ergodicity assumption.
\begin{assumption}\label{assume_ergodic}
For any admissible policy $\pi=\{\mu_0,\mu_1,\ldots,\mu_{\dist-1}\}$ and initial state $x\in \S$, there exist $\pb \in (0,1)$, $\dist >0$ and $\y\in \S$ such that
\begin{equation}
\P_{\pi}(x_{0}=x,x_{\dist}=\y)\defeq
[\P_{a_0}P_{a_1}\ldots \P_{a_{\dist-1}}]_{x\y}\geq \pb, \label{Q_assume}
\end{equation}
where $a_k = \mu_k(x_k)$, $ k = 0,\dots,  \dist-1$. 
\end{assumption}
\section{Weighted Difference Approximation and Its Convergence Properties }\label{sec:VI_conv}
It is well known that there are some intrinsic relationships between maximum average reward and maximum discounted reward MDPs. As discussed in \cite{tsitsiklis2002}, for any admissible control policies, an average reward can be viewed as an orthogonal projection of the discounted reward where the relative value function is a $(1-\df)$ multiple of the residual vector. Furthermore, from Theorem 1 in \cite{kakade2001}, when $\df\rightarrow 1$, the discounted reward can be approximated by maximum average reward.  However, this approximation is valid only when $\df\rightarrow 1$. Also this approach has a major drawback, as finding the optimal control policies (Blackwell optimal control policies) for discounted reward MDPs is usually computationally expensive (cf. Chapter 10 of \cite{puterman_05} for more details). Motivated by these observations, and under Assumption \ref{assume_ergodic}, this section develops a new value function approximation for discounted reward MDPs using weighted difference methods, which also arises in average reward value iteration. We also show that the error bound of this algorithm is geometric and is always smaller than the classical value iteration.

For any specific $z\in \S$,  define the ``gain" $\lambda^\ast$ and the ``bias" $h^\star$ for discounted reward MDPs:
\[
\begin{split}
&h^\star(x)=\V^\star(x)-\V^\star(z),\,\,  \lambda^\star=(1-\df)\V^\star(z).
\end{split}
\]
By Fixed Point theorem: $T[\V^\star](x)=\V^\star(x)$, we have the following identity:
\[
\lambda^\star+h^\star(x)=T[h^\star](x).
\]
This is analogous to the Fixed Point theorem for average reward uni-chain MDPs. Now, we define 
\begin{equation}\label{beta_defn}
\icr=(1-\pb)^{1/\dist}\in(0,1).
\end{equation}
This term can be viewed as an improved discounted factor, and it is well defined, based on the ergodicity assumption (Assumption \ref{assume_ergodic}). More discussions about $\icr$ will be given in the next section.

Now, define the weighted difference value function approximation scheme:
\begin{quote} {\bf $\mathcal{WDVF}$ Approximation Scheme} --- Given an initial value function estimate $V_0:\S\rightarrow\reals$, and a discounted factor $\df\in(0,1)$, for $k\in\{1,2,\ldots\}$, estimate the $(k+1)^{\text{th}}-$step value function as follows:
\begin{equation}\label{val_fn}
\!\!\V_{k+1}(x)\!\!=\!\!\frac{T^{k+1}[\V_0](x)\!\!-\!\!\df T^k[\V_0](x)}{1-\df}, \forall x\in\S.\!\!
\end{equation}\end{quote}

Different from the classical value iteration (which estimates the value function as $T^k[V^0]$ at the $(k+1)^{\text{th}}$ step), the $\mathcal{WDVF}$ approximation uses a normalized one-step difference: $(T^{k+1}[\V_0](x)-\df T^k[\V_0](x))/(1-\df)$ in each updates. If we represent the $k^{\text{th}}-$step value function estimate in classical value iteration by 
\[
\overline{\V}_{k}(x)=T^k[\V_0](x),
\]
the $(k+1)^{\text{th}}-$step $\mathcal{WDVF}$ approximation is equivalent to 
\[
\V_{k+1}(x)=\frac{\overline{\V}_{k+1}(x)-\df\overline{\V}_{k}(x)}{1-\df},\,\, \forall x\in\S,\,\, \alpha\in(0,1).
\]
It is obvious that for any $\df\in(0,1)$, if $\overline{\V}_{k}(x)\rightarrow\overline{\V}_\infty(x)=\V^\star(x)$, then $\V_{k+1}(x)\rightarrow\V^\star(x)$, for any $x\in\S$. In the next theorem, we will show that the error bound of $\mathcal{WDVF}$ approximation converges faster than the error bound of the classical value iteration. Before getting into the details, define the following constant:
\begin{equation}
\dconst=\max_{\ell\in\{0,1,\ldots,\dist-1\}}\frac{\|T^\ell[\V_0]-T^\ell[h^\star]\|_d}{(\df\icr)^\ell}>0\label{dconst_defn}
\end{equation}
where $\|\V\|_d=\max_{x\in \S}\V(x)-\min_{x\in \S}\V(x).\footnote{The $\|\cdot\|_d$ notation is identical to the span-semi norm notation in equation (6.6.3) in \cite{puterman_05}.}$
This constant will characterize the leading coefficient of the error bound in $\mathcal{WDVF}$ algorithm for discounted reward problems, whose explicit formulation is provided in the following theorem.
\begin{theorem}\label{lem_VI_dis}
For $k\in\mathbb{Z}^+$ and any $x\in\S$, let $V_{k+1}(x)$ be the $(k+1)^{\text{th}}-$step $\mathcal{WDVF}$ approximation obtained from equation (\ref{val_fn}). This value function approximation has the following error bound in $\|\cdot\|_d$ semi-norm:
\begin{equation}
\|\V_{k+1}-\V^\star\|_{d}\leq  \frac{\df(1+\icr)(\df\icr)^{ k}}{1-\df} \dconst\label{V_bdd_1}
\end{equation}
and the following error bound for any $x\in\S$:
\begin{equation}
-2\dconst\frac{(\df\icr)^{k}}{1-\df}\leq \V_{k+1}(x)-\V^\star(x)\leq 2\dconst\frac{(\df\icr)^{k}}{1-\df}.\label{V_bdd_2}
\end{equation}
Furthermore, let $c(x)=T^{k-1}[\V_0]-T^k[\V_0]$. Then,
\begin{equation}\label{expression_difference_V}
\frac{\df(c(x)-\|c\|_\infty)}{1-\df}\!\!\leq\!\!\V_{k+1}(x)\!\!-\!\!\V_k(x)\!\!\leq\!\!\frac{\df(c(x)+\|c\|_\infty)}{1-\df}.
\end{equation}
\end{theorem}

\BEF
See appendix.
\end{proof}
\begin{remark}
The difference between any two successive value function estimates in the $\mathcal{WDVF}$ approximation scheme is bounded. However, the sequence of value function is not monotonically increasing/decreasing. 
\end{remark}
\begin{remark}
Similar to the relative value iteration algorithm in Section 6.6.4 in \cite{puterman_05} and in \cite{macqueen1966modified} (which is namely the modified dynamic programming algorithm), the $\mathcal{WDVF}$ approximation is based on the normalized differences between value functions. Thus, these two methods share similar semi-norm convergence rates (the definition of $\gamma$ in Theorem 6.6.6 in \cite{puterman_05} is identical to $\icr$, when $\dist=1$). Nevertheless, our proposed algorithm also has a convergence rate of $(\df\icr)^k$ in sup-norm, while up to the authors' knowledge, no such analysis exists for the relative value iteration algorithm.
\end{remark}

\section{The connection with $\epsilon-$mixing Time}\label{sec:mix_time}
In the previous section, we characterize the error bound of the $\mathcal{WDVF}$ approximation scheme in terms of $\dconst$, $\df$ and $\icr$, where $\dconst$ depends on $\df$, $\icr$ and the value function. The intuition behind the discounted factor $\df$ is very clear. However, based on equation (\ref{beta_defn}), we only know that $\icr$ is related to the ergodicity of a Markov decision process (cf. Section~3 of \cite{morton1977discounting} for details). Its explicit meaning is not well understood. In order to understand the meaning behind $\icr$, it is natural to study the notion of ``$\epsilon-$mixing time" in average reward MDPs. Although we will formally define this notion later, $\epsilon-$mixing time can be viewed as a metric that measures the ``ergodic strength" (the convergence speed of sample average reward function to relative reward function) of average reward MDPs. Intuitively $\epsilon-$mixing time and $\icr$ describe similar features in a Markov decision process.

In this section, we will formulate a relationship between $\icr$ and the $\epsilon-$mixing time. This in turn establishes a connection between the error bound of the $\mathcal{WDVF}$ approximation scheme and $\epsilon-$mixing time.

First, define the Bellman operator for an un-discounted reward function, similar to the case of average reward MDP:
\[
\overline{T}[h](x)=\max_{a\in\A} \sum_{y\in \S}\P_{a}(x,y)(h(x)+R_{a}(x,y)),\quad \forall x\in\S.
\]
Also, define $\Pi$ to be the set of sequence of general admissible policies. The average reward MDP is given by $\max_{\pi\in\Pi} J_\pi(x_0)$,
where
\begin{equation}
J_\pi(x_0)=\lim\sup_{N\rightarrow\infty}\frac{1}{N}\mathbb E \Bigg[\sum_{t=1}^N  R_{\mu_t(x_t)} (x_t, x_{t+1}) \Bigg].
\end{equation}
and $\pi=\{\mu_0,\mu_1,\ldots\}$. From Proposition 5.1.1 and 5.1.2 in \cite{bertsekas_12}, the $``\limsup"$ can be replaced by $``\lim"$ if we restrict $\Pi$ to be the set of stationary admissible policies, i.e., $\pi=\{\mu,\mu,\ldots\}$.

From Section 5.1.3, Proposition 5.1.8 in \cite{bertsekas_12}, for average reward MDP, suppose the relative reward  $\lambda^\star:\S\rightarrow\reals$ and the bias reward $h^\star:\S\rightarrow\reals$ satisfy the following pair of optimality equations:
\begin{subequations}\label{opt_avg_cond_1}
\begin{align}
\lambda^\star(x)=&\max_{a\in\A}\sum_{y\in S} \P_a(x,y)\lambda^\star(y),\\
\!\!\!\!  \lambda^\star(x)\!+\!h^\star(x)\!=\!&\max_{a\in\overline\A}\sum_{y\in S} \P_a(x,y)(R_a(x,y)\!\!+\!\!h^\star(y))\!\!\!
\end{align}
\end{subequations}
where $\overline\A$ is the set of control actions that maximizes the first optimization problem. Then, $\mu^\star$, which attains the maximum of these two expressions simultaneously, is the stationary optimal control policy of the average reward MDP. Furthermore, the following expression holds for any $N^\prime\in\naturals$.
\begin{align}
&\frac{1}{N^\prime}\expectation{\sum_{k=0}^{N^\prime-1}R_{\mu^\star(x_{k})}(x_k,x_{k+1})+h^\star(x_{N^\prime})\mid x_0=x,\mu^\star} \nonumber\\
&-\lambda^\star(x)=h^\star(x)/N^\prime,\quad \forall x\in\S.\label{diff_avg}
\end{align}
Thus, with $h^\star(x)$ being a finite real valued bias function obtained from expression (\ref{opt_avg_cond_1}), by letting $N^\prime\rightarrow\infty$, we can show that $\lambda^\star(x)$ is the optimal average reward:
\[
\lambda^\star(x)=\lim_{N^\prime\rightarrow\infty}\frac{1}{N^\prime}\expectation{\sum_{k=0}^{N^\prime-1}R_{\mu^\star(x_{k})}(x_k,x_{k+1})\mid x_0=x,\mu^\star}.
\]

Consider a stationary policy $\mu$ where the Markov chain induced by $\mu$ only has one recurrent class. We call such stationary policy a uni-chain policy. By proposition 5.2.5 in \cite{bertsekas_12}, if all admissible stationary policies are uni-chain, Assumption \ref{assume_ergodic} holds with $\mu_k=\mu$, for any $k\in\naturals$. Proposition 5.2.3 in \cite{bertsekas_12} implies that the gain $\lambda^\star(x)$ is the same for all states. Then, the first equation in expression (\ref{opt_avg_cond_1}) holds trivially and $\overline \A =\A$. Thus the stationary optimal policy $\mu^\star$ can be found by the following expression:
\[
\mu^\star(x)\in\arg\max_{a\in\A}\sum_{y\in S} \P_a(x,y)(R_a(x,y)+h^\star(x))
\]
and $\lambda^\star$ is the optimal average reward that satisfies the fixed point theorem for average reward MDP:
\[
\lambda^\star+h^\star(x)=\overline{T}[h^\star](x),\quad \forall x\in\S.
\]

Next, the notion of $\epsilon-$mixing time in a MDP is discussed. The standard notion of mixing time of a stationary control policy $\mu$ quantifies the smallest number $N$ of steps required to ensure that the distribution on states after $N$ steps is within $\epsilon$ of the stationary distribution induced by $\mu$. The distance between these distributions is measured by the Kullback-Leibler divergence, the variation distance, or some other standard metrics. There are well-known methods for bounding this mixing time in terms of the second eigenvalue of the transition probability matrix $P$, using underlying structural properties such as ``conductance". Similar to Definition 5 in \cite{Kearns2002}, it turns out that we can state our results for a weaker notion of mixing time that only requires the expected discounted reward after $N$ steps, induced by the stationary optimal control policy to approach an asymptotic reward.
\begin{definition}\label{defn_mix_time}
The $\epsilon-$mixing time  of any stationary optimal control policy, $\mu^\star\in\arg\max_{\mu} V_\mu(x)$, is the smallest constant $\oepmt$ such that for all $N^\prime\geq \oepmt$ and all $x\in\S$,
\begin{equation}
\!\!\!\!\!\!\left|\frac{1}{N^\prime}\expectation{\sum_{k=0}^{N^\prime-1}R_{\mu^\star(x_{k})}(x_k,x_{k+1})\!\!\mid\!\! x_0=x,\mu^\star}\!\!-\!\!\lambda^\star\right|\leq \epsilon.\!\!\label{mixed_cond}
\end{equation}
\end{definition}

Before getting to the main result of this section, we define
\begin{equation}
\aconst=\max_{\ell\in\{0,1,\ldots,\dist-1\}}\frac{\|\overline{T}^\ell[\V_0]-\overline{T}^\ell[h^\star]\|_d}{(1-\pb)^{\ell/\dist}}>0.\label{aconst_defn}
\end{equation}
Similar to the definition of $\dconst>0$, this coefficient will characterize the constant term of an upper bound for average reward problems. The next theorem provides this upper bound in terms of the time horizon $N^\prime$, $\aconst>0$ and $\icr>0$. Also, it gives an expression between $\epsilon-$mixing time and constant $\icr\in(0,1)$.
\begin{theorem}\label{char_beta}
Let $\V_0(x)=0$ for any $x\in\S$. Then, for any $x\in\S$, and for any $N^\prime\geq 1$, there exists a constant $\aconst>0$ such that
\begin{equation}
\begin{split}
&\left|\frac{1}{N^\prime}\expectation{\sum_{k=0}^{N^\prime-1}R_{\mu^\star(x_{k})}(x_k,x_{k+1})\mid x_0=x}-\lambda^\star\right|\\
\leq&\frac{2\aconst}{N^\prime}\frac{\icr}{1-\icr},\quad \forall x\in\S. \label{avg_reward_1}
\end{split}
\end{equation}
where 
\[
\mu^\star(x)\in\arg\max_{a\in\A}\sum_{y\in \S}\P_{a}(x,y)(h(x)+R_{a}(x,y)).
\]
Furthermore, this implies 
\[
\icr\geq\epsilon \oepmt/(2 \aconst+\epsilon \oepmt),
\]
where $\oepmt$ is the $\epsilon-$mixing time in Definition \ref{defn_mix_time}.\footnote{Proof of this result is omitted in this conference version and  can be found at  \url{web.stanford.edu/~ychow}.}
\end{theorem}

\begin{proof}
For any specific $z\in \S$ and $k\in\{1,2,\ldots\}$, define:
\[
\begin{split}
&h_k(x)=\overline{T}^k[\V_0](x)-\overline{T}^k[\V_0](z), \\
& \lambda_k(x)=\overline{T}^k[\V_0](x)-\overline{T}^{k-1}[\V_0](x),\,\, \forall x\in\S.
\end{split}
\]
This implies that 
\[
\lambda_k(x)+h_{k-1}(x)=\overline{T}[h_{k-1}](x).
\]
Recall $\|\V\|_d=\max_{x\in S}\V(x)-\min_{x\in S}\V(x)$.
Similar to the arguments in Lemma \ref{lem_VI_dis} for discounted reward problems, we can show that
\[
\|\overline{T}^{\dist}[\V^{(1)}]-\overline{T}^{\dist}[\V^{(2)}]\|_d\leq(1-\pb) \|\V^{(1)}-\V^{(2)}\|_d.
\]
We can show by induction, and fixed point theorem of average reward MDPs that 
\[
k\lambda^\star+h^\star(x)=\overline{T}^k[h^\star](x).
\]
Moreover, let $k=q\dist+\ell$, for $\ell=\{0,1,\ldots,\dist-1\}$, where $q$ is the greatest common divisor of $k$ and $\dist$. As in in Lemma \ref{lem_VI_dis}, here we can also show that 
\[
\|\overline{T}^k[\V_0]-\overline{T}^k[h^\star]\|_d\leq (\icr)^{qN}\|\overline{T}^\ell[\V_0]-\overline{T}^\ell[h^\star]\|_d\leq \aconst\icr^k.
\]
Following similar derivations as in Lemma \ref{lem_VI_dis}, the above results further imply that $\|h_k-h^\star\|_\infty\leq \aconst\icr^k$
and 
\[
\begin{split}
\|\overline{T}^k[\V_0]-\overline{T}^{k-1}[\V_0]-\lambda^\star\|_\infty&=\|\lambda_k-\lambda^\star\|_\infty\\
&\leq 2\|h_{k}-h^\star\|_\infty\leq 2\aconst\icr^k.
\end{split}
\]

Furthermore, by a telescoping sum, 
\[\small
\begin{split}
&\left|\frac{\overline{T}^N[\V_0](x)-\V_0(x)}{N}-\lambda^\star\right|\leq\sum_{k=1}^N\frac{\|\overline{T}^k[\V_0]-\overline{T}^{k-1}[\V_0]-\lambda^\star\|_\infty}{N}\\
\leq& \frac{2\aconst}{N}\sum_{k=1}^N\icr^k= \frac{2\aconst}{N^\prime}\frac{\icr(1-\icr^{N^\prime})}{1-\icr}\leq \frac{2\aconst}{N^\prime}\frac{\icr}{1-\icr}
\end{split}
\]
for any $x\in\S$. Since $\V_0(x)=0$ for all $x\in S$, the above result implies expression (\ref{avg_reward_1}). Now, for 
$N_0=2\aconst\icr/(\epsilon(1-\icr))$, one obtains
\[
\left|\frac{\expectation{\sum_{k=0}^{N^\prime-1}R_{\mu^\star(x_{k})}(x_k,x_{k+1})\mid x_0=x}}{N^\prime}-\lambda^\star\right|\leq\epsilon
\]
for any $N^\prime\geq N_0$. Then, based on the definition of $\epsilon-$mixing time in Definition \ref{defn_mix_time}, we conclude that 
$2\aconst\icr/(\epsilon(1-\icr))\geq \oepmt$
and $\icr\geq\epsilon \oepmt/(2 \aconst+\epsilon \oepmt)$.
\end{proof}
Now, we are in position to give a relationship between the number of steps needed for convergence of $\mathcal{WDVF}$ approximation and $\epsilon-$mixing time $\oepmt$. Given a constant $\theta>0$. From Lemma \ref{lem_VI_dis}, the condition $\|\V_k-\V^\star\|_\infty\leq\theta$ holds if
\[
\begin{split}
\frac{2\dconst(\df\icr)^{k-1}}{1-\df}\leq \theta\iff &k\geq \log\left(\frac{\theta(1-\df)}{2\dconst}\right)/\log(\df\icr)+1.\end{split}
\]
From the $\icr$ bound given by Theorem \ref{char_beta}, we know that, if the number of steps  is given by the following expression:
\begin{equation}
\begin{split}
k\geq& \ctheta\defeq\max\left\{\frac{\log\left(\theta(1-\df)/(2\dconst)\right)}{\log\left(\df\epsilon \oepmt/(2 \aconst+\epsilon \oepmt)\right)}+1,1\right\},\label{mixing_time_const}
\end{split}
\end{equation}
where $\oepmt$ is the $\epsilon-$mixing time and $\dconst$, $\aconst$ are given by equations (\ref{dconst_defn}) and (\ref{aconst_defn}) respectively,  then $\|\V_k-\V^\star\|_\infty\leq\theta$ is guaranteed. We summarize this result as follows:
\begin{theorem}\label{steps_vi}
Let $V_k(x)$ be the $k^{\text{th}}$ $\mathcal{WDVF}$ approximation obtained from equation (\ref{val_fn}), for any $x\in\S$. The number of steps required for $\|\V_k-\V^\star\|_\infty\leq\theta$ is at least  $\ctheta$.
\end{theorem}

\section{Modified Q-Value Iteration}\label{sec:modified_Q_itr}
In this section, we study the convergence properties of modified $Q-$value iteration.  First, define the following algorithm for modified Q-value iteration: 
\begin{quote} {\bf$\mathcal{WDQVF}$ Approximation Scheme} --- Given an initial value function estimate $V_0:\S\rightarrow\reals$, and a discounted factor $\df\in(0,1)$. Let $Q_0(x,a)$ be the following initial $Q-$function estimate:
\[
Q_0(x,a)=V_0(x),\,\,\forall (x,a)\in\S\times\A.
\]
For $k\in\{1,2,\ldots\}$, update the $(k+1)^{\text{th}}-$step $Q-$function estimate as follows:
\begin{equation}\label{Q_fn}
Q_{k+1}(x,a)=\frac{F^{k+1}[\Q_0](x,a)-\df F^k[\Q_0](x,a)}{1-\df}
\end{equation}
for any $(x,a)\in\S\times\A$.
\end{quote}

By using the error bound result for modified value iteration from the $\mathcal{WDVF}$ approximation, we can prove a similar error bound for $\mathcal{WDQVF}$ approximation. This result is summarized in the following theorem.
 \begin{theorem}
 Let  $\{\Q_k\}$ be a sequence of $\Q$-value function estimates generated by the $\mathcal{WDQVF}$ approximation scheme. Then, the following expression holds for any $(x,a)\in\S\times\A$:
\begin{equation*}
\left|\Q_{k}(x,a)-\Q^\star(x,a)\right|\leq 2\df\dconst\frac{(\df\icr)^{k-2}}{1-\df}=O((\df\icr)^k)
\end{equation*}
\end{theorem}
\begin{proof}
Based on the definitions of $T[\cdot]$ and $F[\cdot]$, we know that 
$\max_{a\in\A}F[\Q_0](x,a)=T[\V_0](x)$.
By repeating the above analysis, we can show by induction that 
$\max_{a\in\A}F^k[\Q_0](x,a)=T^k[\V_0](x),\,\, \forall k\in\naturals.$
We  will use the error bound result in the $\mathcal{WDVF}$ approximation scheme to show a similar error bound for the $\mathcal{WDQVF}$ approximation scheme. First, let 
\[
\overline{\Q}(x,a)=T^k[\V_0](x),\,\, \forall (x,a)\in\S\times\A.
\]
By applying $F[\cdot]$ to the above equation, it implies for any $(x,a)\in\S\times\A$,
\[
\small\begin{split}
&F[T^k[\V_0]](x,a)=F[\overline{Q}](x,a)\\
=&\sum_{y\in\S} P_a(x,y)\left(R_a(x,y)+\df\max_{b\in\A} \overline{Q}(y,b)\right)\\
=&\sum_{y\in\S} P_a(x,y)\left(R_a(x,y)+\df T^k[\V_0](y)\right)\\
=&\sum_{y\in\S}\! P_a(x,y)\!\left(\!\R_a(x,y)\!+\!\df\max_{b\in\A} F^k[\Q_0](y,b)\!\right)\!=\!F^{k+1}[\Q_0](x,a).
\end{split}
\]
Now, expression (\ref{rhs_h}) and (\ref{lhs_h}) imply
\begin{equation*}
\begin{split}
&-\frac{2\|h_{k}-h^\star\|_\infty}{1-\df}+T^k[\V_0](x) \leq \V^\star(x)\\
&-\frac{T^{k}[\V_0](x)- T^{k+1}[\V_0](x)}{1-\df}\leq \frac{2\|h_{k}-h^\star\|_\infty}{1-\df}+T^k[\V_0](x).
\end{split}
\end{equation*}

By applying $F[\cdot]$ to the above inequality, and noting that 
\[
F[\Q+c](x,a)=F[\Q](x,a)+\df c,
\]
we know that for any $(x,a)\in\S\times\A$,
\begin{equation}\label{Q_fact}
\begin{split}
&-\frac{2\df\|h_{k}-h^\star\|_\infty}{1-\df}+F^{k+1}[\Q_0](x,a) \\
\leq& F\left[\V^\star(x)+\frac{T^{k}[\V_0](x)- T^{k+1}[\V_0](x)}{1-\df}\right]\\
\leq &\frac{2\df\|h_{k}-h^\star\|_\infty}{1-\df}+F^{k+1}[\Q_0](x,a)
\end{split}
\end{equation}
Furthermore, by recalling $\max_{a\in\A}\Q^\star(x,a)=\V^\star(x)$, we obtain the following expressions:
\[\small
\begin{split}
 &F\left[\V^\star(x)+\frac{T^{k}[\V_0](x)- T^{k+1}[\V_0](x)}{1-\df}\right]\\
 =&\sum_{y\in\S} P_a(x,y)\bigg(R_a(x,y)+\df \max_{a\in\A} \bigg\{\V^\star(y)\\
 &\qquad\qquad\qquad+\frac{T^{k}[\V_0](y)- T^{k+1}[\V_0](y)}{1-\df}\bigg\}\bigg)\\
 =&\Q^\star(x,a)+\df\sum_{y\in\S} P_a(x,y)\frac{T^{k}[\V_0](y)- T^{k+1}[\V_0](y)}{1-\df}\\
 =&\frac{1}{1-\df}\left(\sum_{y\in\S} P_a(x,y)\left(R_a(x,y)+\df\max_{b\in\A}F^k[\Q_0](y,b)\right)\right.\\
 &\left.-\sum_{y\in\S} P_a(x,y)\!\!\left(\!\!R_a(x,y)\!+\!\df\max_{b\in\A}F^{k+1}[\Q_0](y,b)\right)\!\!\right)\!+\!\Q^\star(x,a)\\
 =&\Q^\star(x,a)+\frac{1}{1-\df}\left(F^{k+1}[\Q_0](x,a)-F^{k+2}[\Q_0](x,a)\right).
 \end{split}
\] 
Thus, by combining all arguments, expression (\ref{Q_fact}) implies 
\begin{equation*}
\begin{split}
&-\frac{2\df\|h_{k}-h^\star\|_\infty}{1-\df} \\
\leq& \Q^\star(x,a)-\left(\frac{F^{k+2}[\Q_0](x,a)}{1-\df}-\frac{\df}{1-\df}F^{k+1}[\Q_0](x,a)\right)\\
\leq &\frac{2\df\|h_{k}-h^\star\|_\infty}{1-\df}
\end{split}
\end{equation*}
Now, by putting the result: $\|h_{k}-h^\star\|_\infty\leq \dconst(\df\icr)^k$ to the above expression, the error bound proof for the $\mathcal{WDQVF}$ approximation scheme is completed.
\end{proof}
\section{Numerical Experiment}\label{sec:sim}
Consider $100$ Monte Carlo samples of randomly generated 100-state-6-action MDPs with $\S=\{1,2,\ldots,100\}$, $\A=\{1,2,3,4,5,6\}$, $\df=0.995$. The reward functions are randomly generated with $R_{\max}=1$. For simplicity each reward function is assumed to be $y-$independent, that is. $R_a(x,y) = R_a(x)$ along $y\in\S$. The transition probabilities induced by each actions are randomly generated with ergodic strength of at least $0.1$ ($\pb=0.1$ and $\dist=1$). This further implies the improved discount factor $\icr$ equals to $0.9$.\footnote{The explicit formulations of the reward functions and transition probabilities can be found in the author's website.} 
We want to compare the performance between the classical value iteration, Gauss-Seidel value iteration and the $\mathcal{WDVF}$ approximation scheme. Recall that the error bound for value iteration is given by ${R_{\max} \df^k}/{(1-\df)}$. From Theorem \ref{lem_VI_dis}, the error bound for $\mathcal{WDVF}$ approximation is given by ${2\dconst (\df\icr)^{k-1}}/{(1-\df)}$. From Proposition 6.3.8 in \cite{puterman_05}, the error bound of Gauss-Seidel value iteration is given by ${R_{\max}(\df\icr^\mathrm{GS})^k}/{(1-\df)}$, where $\icr^\mathrm{GS}<1$ can be calculated using the matrix regular splitting method depicted in Theorem 6.3.4 of \cite{puterman_05}.
\begin{figure}[htb]
\centering
\includegraphics[width=0.35\textwidth]{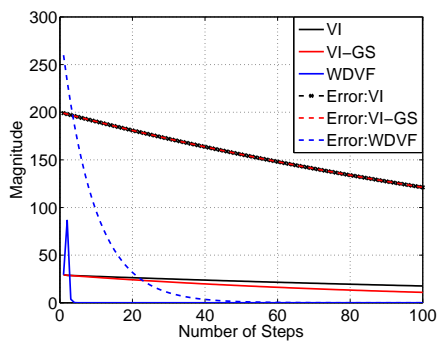}
\caption{Mean of $\|V_k-V^\star\|_\infty$ across Monte Carlo runs.}\label{fig_1}
\end{figure}

Figure \ref{fig_1} compares the error bound and speed of convergence of the $\mathcal{WDVF}$ approximation scheme, Gauss-Seidel value iteration and the classical value iteration. The stopping criterion of this experiment is: $\|V_k-V^\star\|_\infty\leq 10^{-5}$. On average, it is observed that $\mathcal{WDVF}$ approximation takes $92$ iterations (standard deviation: $13$ iterations) to converge, while Gauss-Seidel value iteration and classical value iteration take $2936$ iterations (standard deviation: $197$ iterations) and  $3551$ iterations (standard deviation: $172$ iterations) to converge respectively. As illustrated in Theorem \ref{lem_VI_dis}, the error bound of $\mathcal{WDVF}$ approximation is in the order of $(\df\icr)^k=0.8996^k$, while the error bound of the classical value iteration and Gauss-Seidel value iteration are in the order of $0.995^k$ (as $\df\icr^{\mathrm{GS}}\simeq \df$ numerically in our experiment). This numerical example demonstrates that, when $\df\rightarrow 1$, both classical value iteration and Gauss Seidel may encounter slow convergence issues, while the convergence for the $\mathcal{WDVF}$ approximation depends on $\icr$.

\section{Conclusions and Future Work}\label{sec:conc}
In this paper, we have proposed a novel weighted difference value function estimation scheme for discounted reward MDPs. We have shown that this approximation has an error bound of order $(\df\icr)^k$, $\icr\in(0,1)$, which decays faster than the error bound of classical value iteration  (in order of $\df^k$). We also characterize the improved convergence factor $\icr$ and the speed of convergence of this new approximation using $\epsilon-$mixing time. This characterization explicitly links the convergence speed of weighted difference value function estimation to the system behaviors of the MDP. Furthermore, we also extend the above method to find optimal $Q-$function. The above theoretical result is verified by a numerical experiment. Notice that while Assumption \ref{assume_ergodic} can be justified via Schweitzer's transformation \cite{cavazos1998note} in average reward MDPs, similar transformation does not work under the discounted reward settings. Eliminating the restrictions due to the ergodicity assumption will be left as future work.



\section*{Acknowledgement}
The authors would like to thank Professor Benjamin Van Roy for invaluable discussions.
\bibliographystyle{unsrt} 
\bibliography{ref_dyn_pro}

\addtolength{\abovedisplayskip}{-.02in}
\addtolength{\belowdisplayskip}{-.02in}
\appendix
\noindent{\bf Proof of Theorem~\ref{lem_VI_dis}.}
Let $\V^{(1)}$ and $\V^{(2)}$ be two arbitrary functions that maps $\S$ to $\reals$. We define $\V^{(1)}_{k}(x)=T^{k}[\V^{(1)}](x)$ and $\V^{(2)}_{k}(x)=T^{k}[\V^{(2)}](x)$ for any $x\in\S$ and for $k\in\{0,\ldots,\dist\}$. We also define two sequences of optimal policies, $\pi^{(j)}=\{\mu_0^{(j)},\mu_1^{(j)},\ldots,\}$, for $j\in\{1,2\}$, where
$\mu_k^{(j)}(x)\in\arg\max_{a\in\A}\sum_{y\in \S} \P_a(x,y)\left( R_a(x,y) + \df \V_k^{(j)}(y)\right).$ 
For any sequence of state feedback control policies $\pi=\{\mu_0,\mu_1,\ldots,\}$, define the following event: $\hist(x,\pi)=\{x_0=x,a_i=\mu_i(x_i),\,\forall i\}$, where $\{x_j\}_{j\in\mathbb Z^+}$ is a Markov chain induced by control policy $\pi$ with $x_0=x$. By substituting the sequences of optimal policies to value function $V^{(j)}_{\dist}$, one notices that for $j\in\{1,2\}$, and for any $x\in \S$,
\[
\begin{split}
\!\!V^{(j)}_{\dist}(x)\!\!=\!\!
\expectation{\sum_{i=0}^{\dist\!-1}\!\!\df^i R_{a_i}\!(x_i,x_{i+1})\!\!+\!\!\df^{\dist}\! \V^{(j)}(x_{\dist}\!)\!\!\mid\!\! \hist(x,\pi^{(j)})\!}.
\end{split}
\]
By considering the difference between $\V^{(1)}_{\dist}(x)$ and $\V^{(2)}_{\dist}(x)$, we get
\[
\begin{split}
&\V^{(1)}_{\dist}(x)-\V^{(2)}_{\dist}(x)\\
&= \expectation{\sum_{i=0}^{\dist-1}\!\!\df^i R_{a_i}(x_i,x_{i+1})\!+\!\df^{\dist} \V^{(1)}(x_{\dist})\mid \hist(x,\pi^{(1)})}\\
&- \expectation{\sum_{i=0}^{\dist-1}\!\!\df^i R_{a_i}(x_i,x_{i+1})\!+\!\df^{\dist} \V^{(2)}(x_{\dist})\mid \hist(x,\pi^{(2)})}\\
&\geq  \expectation{\sum_{i=0}^{\dist-1}\!\!\df^i R_{a_i}(x_i,x_{i+1})\!+\!\df^{\dist} \V^{(1)}(x_{\dist})\mid \hist(x,\pi^{(2)})}\\
&- \expectation{\sum_{i=0}^{\dist-1}\!\!\df^i R_{a_i}(x_i,x_{i+1})\!+\!\df^{\dist} \V^{(2)}(x_{\dist})\mid \hist(x,\pi^{(2)})}\\
&=\expectation{\df^{\dist} (\V^{(1)}(x_{\dist})-\V^{(2)}(x_{\dist}))\mid \hist(x,\pi^{(2)})}.
\end{split}
\]
The first inequality is due to the fact that for any $k\in\mathbb Z^+$, $\mu_k^{(2)}(x)$ is a feasible solution to the optimization problem
$\max_{a\in\A}\sum_{y\in \S} \P_a(x,y)\left( R_a(x,y) + \df \V_k^{(1)}(y)\right),$
where $\mu_k^{(1)}(x)$ is an optimal solution of this problem, for every $x\in\S$. By Assumption \ref{assume_ergodic}, this further implies that
\[
\begin{split}
&\big(\V^{(1)}_{\dist}(x)-\V^{(2)}_{\dist}(x)\big)/{\df^{\dist}}\\
\geq&\sum_{y\in S}\probj_{\pi^{(2)}}(x_{0}=x,x_{\dist}=y)(\V^{(1)}(y)-\V^{(2)}(y))\\
\geq&[(1-\pb)\!\min_{y\in \S}\{\V^{(1)}(y)\!-\!\V^{(2)}(y)\}\!+\!\pb(\V^{(1)}(\y)\!-\!\V^{(2)}(\y))],
\end{split}
\]
where $\y\in\S$ is the state defined in Assumption \ref{assume_ergodic}.
Similarly, by a symmetric argument, we can also prove that
\[
\begin{split}
&\frac{1}{\df^{\dist}}\max_{y\in\S}\left\{T^{\dist}[\V^{(1)}](y)-T^{\dist}[\V^{(2)}](y)\right\}\\
\leq&[(1-\pb)\!\max_{y\in \S}\{\V^{(1)}(y)\!-\!\V^{(2)}(y)\}\!+\!\pb(\V^{(1)}(\y)\!-\!\V^{(2)}(\y))].
\end{split}
\]
Thus, by these inequalities and the definitions of $\|T^{\dist}[\V^{(1)}]-T^{\dist}[\V^{(2)}]\|_d$, $\|\V^{(1)}-\V^{(2)}\|_d$, we can show that the following $\dist-$step contraction property holds:
\[
\|T^{\dist}[\V^{(1)}]-T^{\dist}[\V^{(2)}]\|_d\leq(\df\icr)^{\dist}\|\V^{(1)}-\V^{(2)}\|_d.
\]

By mathematical induction and the definitions of $\lambda^\star$, $h^\star$, it can be easily shown that
\begin{equation}\label{MI_bellman}
 \sum_{i=0}^{k-1}\df^i\lambda^\star+h^\star(x)=T^k[h^\star](x), \quad \forall x\in\S.
 \end{equation}
Consider the expression: $\left\|T^k[\V_0]-\sum_{i=0}^{k-1}\df^i\lambda^\star-h^\star\right\|_d$. By writing $k=q\dist+\ell$, $\ell=\{0,1,\ldots,\dist-1\}$, where the nonnegative integer $q$ is the greatest common divisor of $k$ and $\dist$, from expression (\ref{MI_bellman}), we obtain the following relationship:
\begin{equation}
\begin{split}
&\big\|T^k[\V_0]-\sum_{i=0}^{k-1}\df^i\lambda^\star-h^\star\big\|_d=\|T^k[\V_0]-T^k[h^\star]\|_d\\
&\leq(\df\icr)^{q\dist}\|T^\ell[\V_0]-T^\ell[h^\star]\|_d\leq \dconst(\df\icr)^k.\label{alpha_beta_con}
\end{split}
\end{equation}
Note that $T^k[h^\star](x)=T^k[\V^\star](x)-\df^k\V^\star(z)=\V^\star(x)-\df^k\V^\star(z)$. From Section 6.6.1 in \cite{puterman_05}, one also obtains $\|u+v\|_d\leq \|u\|_d+\|v\|_d$, $\|-u\|_d=\|u\|_d$, $\|ku\|_d= |k|\|u\|_d$ and $\|u+k\|_d=\|u\|_d$ for any scalar $k$. Therefore, the above expression implies
$
\|T^k[\V_0]-T^k[h^\star]\|_d=\|T^k[\V_0]-\V^\star\|_d\leq \dconst(\df\icr)^k
$
and
\[
\begin{split}
&\|\V_{ k+1}-\V^\star\|_d=\left\|\frac{( T^{ k+1}[\V_0]-\V^\star)-\df( T^{ k}[\V_0]-\V^\star)}{1-\df}\right\|_d\\
\leq &  \left(\| T^{ k+1}[\V_0]-\V^\star\|_d+\df\| T^{ k}[\V_0]-\V^\star\|_d \right)/({1-\df})\\
\leq &(\df\icr)^{ k}\left((\df\icr)+\df\right)\dconst/({1-\df} ).
\end{split}
\]
This implies that the error bound in expression (\ref{V_bdd_1}) holds.

Next, we will show the error bound in expression (\ref{V_bdd_2}).
Define the following quantities that estimate the gain and bias in the $k^{\text{th}}$ step:
\begin{equation*}
\begin{split}
h_k(x)=&T^k[\V_0](x)-T^k[\V_0](z),\\
\lambda_k(x)=&T^k[\V_0](x)-T^{k-1}[\V_0](x)+(1-\df)T^{k-1}[\V_0](z)
\end{split}
\end{equation*}
where $z\in\S$ is an arbitrary reference state.
By simple calculations, the above expressions imply 
$\lambda_{k+1}(x)+h_{k}(x)=T[h_{k}](x).$
It can be easily seen that $h^\star(z)=\V^\star(z)-\V^\star(z)=0$ and
\[
\begin{split}
&|h_k(x)-h^\star(x)|=|T^k[\V_0](x)-T^k[\V_0](z)-h^\star(x)+h^\star(z)|\\
&\!=\!\big|\!T^k[\V_0](x)\!\!-\!\!\sum_{i=0}^{k-1}\!\df^i\lambda^\star\!\!-\!\!h^\star(x)\!\!\!-\!\!\big[T^k[\V_0](z)\!\!-\!\!\sum_{i=0}^{k-1}\!\df^i\lambda^\star\!\!-\!\!h^\star(z)\big]\!\big|\\
&\leq\big\|T^k[\V_0]-\sum_{i=0}^{k-1}\df^i\lambda^\star-h^\star\big\|_{d}\leq \dconst(\df\icr)^k.
\end{split}
\]
Thus, the above inequality implies
\[
\|h^\star-h_k\|_\infty=\max_{x\in\S}|h^\star(x)-h_k(x)|\leq \dconst(\df\icr)^k.
\]
Next, we know from the contraction property of $T[\cdot]$ that
\[
\begin{split}
&\lambda_{k+1}(x)+h_{k}(x)-(\lambda^\star+h^\star(x))=T[h_{k}](x)-T[h^\star](x)\\
\leq&\max_{b\in \A}\df\sum_{y\in \S} P_b(x,y)|h_{k}(y)-h^\star(y)|\leq \df\|h_{k}-h^\star\|_\infty.
\end{split}
\]
By using the definitions of $\lambda_{k+1}(x)$, $h_{k}(x)$, $\lambda^\star$ and $h^\star(x)$, the above expression implies
\[
\begin{split}
T^{k+1}[\V_0](x)-&T^k[\V_0](x)+(1-\df)(T^k[\V_0](z)-\V^\star(z))\\
&+h_{k}(x)-h^\star(x)\leq \df\|h_{k}-h^\star\|_\infty,
\end{split}
\]
which further implies
\[
\begin{split}
&T^{k+1}[\V_0](x)-T^k[\V_0](x)+(1-\df)(T^k[\V_0](z)-\V^\star(z))\\
\leq& (1+\df)\|h_{k}-h^\star\|_\infty.
\end{split}
\]
By inserting 
\[
T^k[\V_0](z)-\V^\star(z)=T^k[\V_0](x)-h_{k}(x)-(\V^\star(x)-h^\star(x))
\]
to the above expression, we get
\[
\begin{split}
T^{k+1}[\V_0]&(x)-T^k[\V_0](x)+(1-\df)(T^k[\V_0](x)-\V^\star(x))\\
&-(1-\df)(h_{k}(x)-h^\star(x))\leq (1+\df)\|h_{k}-h^\star\|_\infty.
\end{split}
\]
This implies
\[
\begin{split}
&T^{k+1}[\V_0](x)-T^k[\V_0](x)+(1-\df)(T^k[\V_0](x)-\V^\star(x))\\
\leq& (1+\df)\|h_{k}-h^\star\|_\infty+(1-\df)(h_{k}(x)-h^\star(x))\\
\leq& 2\|h_{k}-h^\star\|_\infty.
\end{split}
\]
By combining all inequalities, we get
\begin{equation}
\!\!\!\!\!\!\frac{T^{k+1}[\V_0](x)-\df T^k[\V_0](x)}{1-\df}-\V^\star(x)\leq \frac{2\|h_{k}-h^\star\|_\infty}{1-\df}.\label{rhs_h}\!\!
\end{equation}
Similarly, by noting that
\[
\begin{split}
&T[h_{k}](x)-T[h^\star](x)\\
\geq&-\df\max_{b\in \A}\sum_{y\in \S} P_a(x,y)|h_{k}(y)-h^\star(y)|\geq -\df\|h_{k}-h^\star\|_\infty
\end{split}
\]
and applying analogous arguments as in the derivation of inequality (\ref{rhs_h}), we get
\begin{equation}
\!\!\!\!\frac{T^{k+1}[\V_0](x)\!-\!\df T^k[\V_0](x)}{1-\df}\!-\!\V^\star(x)\!\geq \!-\!2\frac{\|h_{k}\!-\!h^\star\|_\infty}{1-\df}.\label{lhs_h}\!\!
\end{equation}
Now, since 
$\|h_{k}-h^\star\|_\infty\leq \dconst(\df\icr)^k,$
the definition of $V_{k+1}(x)$, expression (\ref{rhs_h}) and (\ref{lhs_h}) imply expression (\ref{V_bdd_2}) holds for all $x\in\S$. This completes the first part of the proof.

Finally, we will show expression (\ref{expression_difference_V}) holds. For any $k\in\mathbb Z^+$, the $\mathcal{WDVF}$ approximation can be re-written as
\[
\begin{split}
&\V_{k+1}(x)=\frac{T^{k+1}[\V_0](x)-T^{k}[\V_0](x)}{1-\df}+T^k[\V_0](x).
\end{split}
\]
Thus, we know that
\[\small
\begin{split}
\V_{k+1}&(x)-\V_k(x)=\frac{T^{k+1}[\V_0](x)+T^{k-1}[\V_0](x)-2T^{k}[\V_0](x)}{1-\df}\\
&\qquad\qquad\qquad+T^k[\V_0](x)-T^{k-1}[\V_0](x)\\
\leq&\frac{\df(\|T^k[\V_0]-T^{k-1}[\V_0]\|_\infty+T^{k-1}[\V_0](x)-T^{k}[\V_0](x))}{1-\df}.
\end{split}
\]
The first inequality is implied by the fact that $T[\cdot]$ is a $\df-$contraction mapping:
\[
T^{k+1}[\V_0](x)-T^{k}[\V_0](x)\!\leq \!\df\|T^{k}[\V_0]-T^{k-1}[\V_0]\|_\infty,\,\, \forall x\in\S.
\]
On the other hand, we can also show that
\[
\begin{split}
&\V_{k+1}(x)-\V_k(x)\\
\geq&\frac{\df(-\|T^k[\V_0]-T^{k-1}[\V_0]\|_\infty+T^{k-1}[\V_0](x)-T^{k}[\V_0](x))}{1-\df}.
\end{split}
\]
by analogous arguments. This completes the second part of the proof.

\end{document}